\documentclass[11pt]{article}
\usepackage[utf8]{inputenc}
\usepackage{amssymb,amsmath,amsfonts,amsthm,mathrsfs}
\usepackage{graphicx,graphics,psfrag,epsfig,cancel}
\usepackage[colorlinks=true, pdfstartview=FitV, linkcolor=blue, citecolor=red, urlcolor=blue]{hyperref}

\usepackage{caption,subfig,color}
\usepackage{tikz}
\usetikzlibrary{calc}

\usepackage{geometry}
\newtheorem{assumption}{Assumption}

\newtheorem{lemma}{Lemma}

\newtheorem{theorem}{Theorem}
\newtheorem{corollary}{Corollary}

\newcommand{\N}{{\mathbb N}}
\newcommand{\Z}{{\mathbb Z}}
\newcommand{\R}{{\mathbb R}}
\newcommand{\C}{{\mathbb C}}

\newcommand{\cercle}{{\mathbb S}^1}

\newcommand{\bqs}{\begin{equation*}}
\newcommand{\eqs}{\end{equation*}}
\newcommand{\bqq}{\begin{equation}}
\newcommand{\eqq}{\end{equation}}
\renewcommand{\Re}{\mathrm{Re}}
\renewcommand{\Im}{\mathrm{Im}}

\title{The local limit theorem for complex valued sequences: \\ 
the parabolic case}

\author{Jean-Fran\c{c}ois {\sc Coulombel} \& Gr\'egory {\sc Faye}\thanks{Institut de Math\'ematiques de Toulouse - UMR 5219, Universit\'e de Toulouse ; CNRS, Universit\'e Paul Sabatier, 118 route de Narbonne, 31062 Toulouse Cedex 9 , France. Research of G.F. acknowledges support from the ANR via the project Indyana under grant agreement ANR- 21- CE40-0008 and from Labex CIMI under grant agreement ANR-11-LABX-0040. Emails: {\tt jean-francois.coulombel@math.univ-toulouse.fr}, {\tt gregory.faye@math.univ-toulouse.fr}}}
\date{\today}

\begin{document}

\maketitle

\begin{abstract}
We give a complete expansion, at any accuracy order, for the iterated convolution of a complex valued integrable sequence in one space dimension. The remainders are estimated sharply with gene\-ralized Gaussian bounds. The result applies in probability theory for random walks as well as in numerical analysis for studying the large time behavior of numerical schemes.
\end{abstract}

\section{Introduction}
\label{section1}

The local limit theorem in probability theory \cite[Chapter VII]{Petrov} gives an asymptotic expansion of the probability:
$$
\mathbb{P}(X_1+\cdots+X_n=j)
$$
where $X_1,...,X_n,...$ are independent, identically distributed, random variables with values in $\Z$. Some versions of the local limit theorem even consider non identically distributed random variables. Usually, the expansion is understood in the sense that the discrete time $n$ becomes large, and one wishes to obtain remainders that are, at least, uniform with respect to the position $j \in \Z$. The terms in the expansion correspond to increasing powers of $n^{-1/2}$, and the leading term in the expansion is a suitably scaled Gaussian function, as can be expected from  the  central  limit  theorem. As a matter of fact, summing with respect to $j$ the leading term in the expansion, we expect to recover in the expression:
$$
\mathbb{P}(X_1+\cdots+X_n \le J) \, = \, \sum_{j=-\infty}^J \, \mathbb{P}(X_1+\cdots+X_n=j)
$$
an approximation by some kind of Riemann sum of the cumulative distribution function of a normal distribution (evaluated at some well-chosen point).

When all random variables  are assumed to be independent and identically distributed, the probability $\mathbb{P}(X_1+\cdots+X_n=j)$ corresponds to the value at the index $j$ of the iterated convolution $(n-1)$-times of the sequence $(\mathbb{P}(X_1=\ell))_{\ell \in \Z}$ with itself. Following, among others, the series of works \cite{Thomee,hedstrom,Despres1,Diaconis-SaloffCoste,RSC1,RSC2,CF1,Coeuret,Randles}, we aim here at giving a complete asymptotic expansion of such iterated convolutions without making any positivity assumption, that is by going beyond the probabilistic framework. The article \cite{Diaconis-SaloffCoste} gives a large overview of examples where this issue is meaningful. As explained and evidenced in \cite{RSC1}, dropping the positivity assumption yields a much larger variety of possible behaviors that correspond, in the language of partial differential equations, either to \emph{parabolic} or \emph{dispersive} behaviors. The present work focuses on the parabolic case, which is the \emph{stable} case in \cite{Thomee} (the case in which the iterated convolutions will be bounded in the $\ell^1$ norm). The dispersive case will be dealt with in a subsequent work.

The results in the above mentioned references contained either technical restrictions on the Fourier transform\footnote{This function is referred to as the \emph{characteristic function} in probability theory, see \cite{Petrov}.} of the considered sequence or did not provide sharp enough estimates for the remainders so that they could be used, e.g., to give error estimates for numerical analysis purposes (see for instance Corollary \ref{coro1} below). In this article, we drop all previous technical restrictions and give an asymptotic expansion up to any order with a sharp, generalized Gaussian estimate for the remainders. The latter estimates yield optimal large time decay estimates for numerical schemes. An example of a third order finite difference approximation of the transport equation is detailed in order to illustrate our main result.

\paragraph{Notation.} In all this article, we let $B(z,\delta)$ denote the open disk in the complex plane $\{ w \in \C \, | \, |z-w|<\delta \}$ that is centered at $z$ and has radius $\delta>0$. We also let $\mathcal{C}(z,\delta)$ denote the open square $\{ w \in \C \, | \, \max \, (|\Re \, (z-w)|,|\Im \, (z-w)|<\delta \}$ in the complex plane that is centered at $z$ and whose side length equals $2 \, \delta$. Eventually we let $\cercle$ denote the unit circle in $\C$. Any other notation is meant to be self-explanatory or is introduced in the core of the article.

\section{Assumptions and main result}
\label{section2}

For any two complex valued sequences $\boldsymbol{a}=(a_\ell)_{\ell \in \Z}$ and $\boldsymbol{b}=(b_\ell)_{\ell \in \Z}$ such that the quantity below makes sense, the convolution $\boldsymbol{a} \star \boldsymbol{b}$ of $\boldsymbol{a}$ and $\boldsymbol{b}$ is defined as:
$$
\boldsymbol{a} \star \boldsymbol{b} := \left( \sum_{\ell' \in \Z} a_{\ell-\ell'} \, b_{\ell'} \right)_{\ell \in \Z} \, .
$$
For instance, the celebrated Young's inequality shows that, for $\boldsymbol{a}$ and $\boldsymbol{b}$ in the space of complex valued integrable sequences $\ell^1(\Z;\C)$, the convolution $\boldsymbol{a} \star \boldsymbol{b}$ is well defined and also belongs to $\ell^1(\Z;\C)$, which endows this space with a Banach algebra structure. The goal of this article is to study the geometric sequences (at least, some of them) in this algebra. In all this article, we consider a fixed complex valued sequence $\boldsymbol{a}=(a_\ell)_{\ell \in \Z}$ and we make the following assumption.

\begin{assumption}
\label{ass:1}
The sequence $\boldsymbol{a}=(a_\ell)_{\ell \in \Z}$ belongs to $\ell^1(\Z;\C)$ and its associated Fourier series:
$$
F_{\boldsymbol{a}} \quad : \quad \zeta \in \C \longmapsto \sum_{\ell \in \Z} \, a_\ell \, \zeta^\ell \, ,
$$
defines a holomorphic function on an annulus $\{ \zeta \in \C \, | \, 1-\varepsilon < |\zeta| <1+\varepsilon \}$ for some $\varepsilon>0$. Furthermore, there holds:
$$
\sup_{\kappa \in \cercle} \, |F_{\boldsymbol{a}}(\kappa)| \, = \, 1 \, .
$$
\end{assumption}

The latter normalization for the maximum of $|F_{\boldsymbol{a}}|$ on the unit circle is made in order to avoid introducing additional terms in the main result below. In numerical analysis, this normalization corresponds to the von Neumann stability condition \cite{LaxRichtmyer}. Fixing the maximum to $1$ can always be achieved up to multiplying the considered sequence ${\boldsymbol{a}}$ by some positive number. Thanks to Cauchy's formula \cite{rudin}, the holomorphy of $F_{\boldsymbol{a}}$ on an annulus that contains the unit circle $\cercle$ is equivalent to the existence of a positive constant $c$ such that:
\begin{equation}
\label{borneexp}
\sup_{\ell \in \Z} \, {\rm e}^{c \, |\ell|} \, |a_\ell| <+\infty \, .
\end{equation}
We now recall the following alternative that has already been proved in our former work \cite{CF1} (see also \cite[page 98]{BTW}):

\begin{lemma}[Lemma A.1 in \cite{CF1}]
\label{lem:comportementF}
Let the sequence ${\boldsymbol{a}}$ satisfy Assumption \ref{ass:1}. Then one of the following is satisfied:
\begin{itemize}
 \item $F_{\boldsymbol{a}}(\kappa)$ has modulus $1$ for any $\kappa \in \cercle$ (e.g., $F_{\boldsymbol{a}}$ is a Blaschke product \cite{rudin}),
 \item there exists a finite set of pairwise distinct points $\{ \underline{\kappa}_1,\dots,\underline{\kappa}_K \}$, $K \ge 1$, in $\cercle$ such that  $F_{\boldsymbol{a}}(\underline{\kappa}_k)$ has modulus $1$ for any $k \in \{ 1,\dots,K \}$ and:
$$
\forall \, \kappa \in \cercle \setminus \big\{ \underline{\kappa}_1,\dots,\underline{\kappa}_K \big\} \, ,\quad \big| F_{\boldsymbol{a}}(\kappa) \big| \, < \, 1 \, .
$$ 
\end{itemize}
\end{lemma}

We explore in this article the behavior of the iterated convolutions ${\boldsymbol{a}} \star \cdots \star {\boldsymbol{a}}$ when ${\boldsymbol{a}}$ satisfies Assumption \ref{ass:1} and also satisfies the second possibility in Lemma \ref{lem:comportementF}. We refer below to the $F_{\boldsymbol{a}}(\underline{\kappa}_k)$'s as to tangency points since these are the points where the curve\footnote{Because of Assumption \ref{ass:1}, this curve is located inside the closed unit disk.} $\{ F_{\boldsymbol{a}}(\kappa) \, | \, \kappa \in \cercle \}$ meets the unit circle $\cercle$. The reader interested in the first case of Lemma \ref{lem:comportementF} may consult \cite{hedstrom} and \cite[Theorem 3.1]{BTW}. We shall make the following crucial assumption in what follows.

\begin{assumption}
\label{ass:2}
The sequence ${\boldsymbol{a}}$ satisfies Assumption \ref{ass:1} and its Fourier series $F_{\boldsymbol{a}}$ satisfies the second possibility in Lemma \ref{lem:comportementF}. Moreover, at any point $\underline{\kappa}_k \in \cercle$, $k \in \{ 1,\dots,K \}$, where the modulus of $F_{\boldsymbol{a}}$ attains the value $1$, there exists a real number $\alpha_k$, a complex number $\beta_k$ with positive real part and a nonzero integer $\mu_k \in \N^*$ such that, as the complex number $\xi$ tends to zero, there holds:
\begin{equation}
\label{DLFa}
F_{\boldsymbol{a}} \Big( \underline{\kappa}_k \, {\rm e}^{\mathbf{i} \, \xi} \Big) \, = \, 
F_{\boldsymbol{a}}(\underline{\kappa}_k) \, \exp \Big( \mathbf{i} \, \alpha_k \, \xi 
-\beta_k \, \xi^{2\, \mu_k} +O(\xi^{2\, \mu_k+1}) \Big) \, .
\end{equation}
\end{assumption}

Let us fix some more notation. First, the iterated convolution ${\boldsymbol{a}}^{\star n}$ is defined by ${\boldsymbol{a}}^{\star 1}:={\boldsymbol{a}}$ and for any $n \in \N^*$, ${\boldsymbol{a}}^{\star (n+1)}:={\boldsymbol{a}}^{\star n} \star {\boldsymbol{a}}$, which corresponds to the geometric sequence associated with ${\boldsymbol{a}}$ in the algebra $\ell^1(\Z;\C)$. Then, following \cite{RSC1}, for any nonzero integer $\mu \in \N^*$ and for any complex number $\beta$ with positive real part, we introduce the function:
\begin{equation}
\label{defHmubeta}
H^\beta_{2 \, \mu} \, : \, x \in \R \longmapsto \dfrac{1}{2 \, \pi} \int_{\R} {\rm e}^{-\mathbf{i} \, x \, \theta} \, {\rm e}^{-\beta \, \theta^{2\, \mu}} \, {\rm d}\theta \, .
\end{equation}
These functions (referred to as \emph{attractors} in \cite{RSC1}) will play a major role in the asymptotic expansions of this article. Some of their basic properties are recalled later on. We now try, as much as possible, to stick to the notation in \cite{Petrov}. Using Assumptions \ref{ass:1} and \ref{ass:2}, we consider a point $\underline{\kappa}_k \in \cercle$ at which $F_{\boldsymbol{a}}$ has modulus $1$. Up to using the logarithm, for any sufficiently small $\xi \in \C$, we can write $F_{\boldsymbol{a}}(\underline{\kappa}_k \, {\rm e}^{\mathbf{i} \, \xi})$ as the convergent power series:
\begin{equation}
\label{DLk}
F_{\boldsymbol{a}} \Big( \underline{\kappa}_k \, {\rm e}^{\mathbf{i} \, \xi} \Big) \, = \, 
F_{\boldsymbol{a}}(\underline{\kappa}_k) \, \exp \left( \mathbf{i} \, \alpha_k \, \xi -\beta_k \, \xi^{2\, \mu_k} 
+\sum_{\nu \ge 2\, \mu_k+1} \dfrac{\gamma_{k,\nu}}{\nu \, !} \, (\mathbf{i} \, \xi)^\nu \right) \, .
\end{equation}
The coefficients $\gamma_{k,\nu}$ play the role of \emph{cumulants} in probability theory. Starting from the power series expansion \eqref{DLk}, we follow \cite{Petrov} and expand a power series in two variables $(Y,Z)$ as follows:
\begin{equation}
\label{defPknu}
\exp \left( \sum_{\nu \ge 1} \dfrac{\gamma_{k,2\, \mu_k+\nu}}{(2\, \mu_k+\nu) \, !} \, 
Y^{2\, \mu_k+\nu} \, Z^\nu \right) \, = \, 1+\sum_{m \ge 1} \, P_{k,m}(Y) \, Z^m \, ,
\end{equation}
where the $P_{k,m}$'s are polynomials with complex coefficients that depend on the cumulants $\gamma_{k,\nu}$ (see several formulas below based on the Fa\`a di Bruno formula \cite{Comtet}). With the above notation, our main result reads as follows.

\begin{theorem}
\label{thm1}
Let the sequence ${\boldsymbol{a}}$ satisfy Assumptions \ref{ass:1} and \ref{ass:2}. Then there exist an integer $L \in \N^*$ and some positive constant $c_0>0$ such that for any $n \in \N^*$ and $\ell \in \Z$ with $|\ell|>L \, n$, there holds:
\begin{equation}
\label{estim1}
\left| {\boldsymbol{a}}^{\star n}_\ell \right| \le \exp (-c_0 \, n -c_0 \, |\ell|) \, .
\end{equation}
Moreover, for any integer $M \in \N$, there exist some positive constants $C_M$ and $c_M$ (that depend on $M$ and ${\boldsymbol{a}}$) such that the following holds: for any $n \in \N^*$ and $\ell \in \Z$ with $|\ell| \le L \, n$, there holds:
\begin{multline}
\label{estim}
\left| {\boldsymbol{a}}^{\star n}_\ell -\sum_{k=1}^K \dfrac{\underline{\kappa}_k^{-\ell} F_{\boldsymbol{a}}(\underline{\kappa}_k)^n}{n^{1/(2\, \mu_k)}} \, H^{\beta_k}_{2 \, \mu_k} \left( \dfrac{\ell-\alpha_k n}{n^{1/(2\, \mu_k)}} \right)
-\sum_{k=1}^K \sum_{m=1}^M \dfrac{\underline{\kappa}_k^{-\ell} F_{\boldsymbol{a}}(\underline{\kappa}_k)^n}{n^{(m+1)/(2\, \mu_k)}} \, \Big( P_{k,m} (-{\rm d}/{\rm d}x) H^{\beta_k}_{2 \, \mu_k} \Big) \left( \dfrac{\ell-\alpha_k n}{n^{1/(2\, \mu_k)}} \right) \right| \\
\le C_M \, \sum_{k=1}^K 
\dfrac{1}{n^{(M+2)/(2\, \mu_k)}} \, \exp \left( -c_M \, \left( \dfrac{|\ell-\alpha_k \, n|}{n^{1/(2\, \mu_k)}} \right)^{\frac{2\, \mu_k}{2\, \mu_k-1}} \right) \, ,
\end{multline}
where the polynomials $P_{k,m}$ are defined in \eqref{defPknu}.
\end{theorem}

Let us make several comments on Theorem \ref{thm1}. In probability theory, the sequence ${\boldsymbol{a}}$ is given by $a_\ell=\mathbb{P}(X_1=\ell)$. It contains only non-negative real numbers that sum to $1$. Assumption \ref{ass:1} is therefore satisfied if and only if the $a_\ell$'s satisfy a bound of the form \eqref{borneexp} for some constant $c>0$. In that case, the first scenario in Lemma \ref{lem:comportementF} is possible if and only if one single $a_\ell$ equals $1$ and all other are zero. This scenario corresponds to a deterministic random walk where, at each time step, one makes a translation of a fixed number $\ell_0$ on the grid $\Z$. Let us therefore assume that the sequence $a$ possesses at least two nonzero elements and that, for simplicity\footnote{Some rare cases yield $K=2$, take for instance $a_{-1}=a_1=1/2$.}, $F_{\boldsymbol{a}}$ has modulus $1$ only at $\kappa=1$ for $\kappa \in \cercle$. This corresponds to $K=1$ in the notation of Assumption \ref{ass:2}, with $\underline{\kappa}_1=1$ and $F_{\boldsymbol{a}}(\underline{\kappa}_1)=1$. Moreover, $\alpha_1$ is nothing but the mean of the random variable $X_1$, $\mu_1=1$ and $2 \, \beta_1=\sigma_1^2>0$ is the variance of $X_1$. From the definition \eqref{defHmubeta}, we compute:
$$
\forall \, x \in \R \, ,\quad H_2^{\beta_1}(x) 
=\dfrac{1}{\sqrt{4\pi \beta_1}} \, \exp \left( -\dfrac{x^2}{4\, \beta_1} \right) \, .
$$
The asymptotic expansion provided by Theorem \ref{thm1} reads:
$$
{\boldsymbol{a}}^{\star n}_\ell \sim \dfrac{1}{\sqrt{4\pi \beta_1 n}} 
\exp \left( -\dfrac{(\ell-\alpha_1 n)^2}{4 \beta_1 n} \right) 
+\sum_{m \ge 1} \dfrac{1}{n^{(m+1)/2}} \, \Big( P_{1,m} (-{\rm d}/{\rm d}x) H^{\beta_1}_2 \Big) \left( \dfrac{\ell-\alpha_1 n}{\sqrt{n}} \right) \, .
$$
It can be rewritten, as in \cite[Chapter VII]{Petrov}, in terms of Hermite polynomials by using the properties of the Gaussian function (the so-called Rodrigues formula for Hermite polynomials, see \cite{Szego}). For instance, we shall see below that the polynomial $P_{1,1}$ is given by:
$$
P_{1,1}(X) \, = \, \dfrac{\gamma_{1,3}}{3\, !} \, X^3 \, ,
$$
where $\gamma_{1,3}$ is the cumulant of order $3$ at the zero frequency, see \eqref{DLk}. This means that the two first terms ($M=1$) in the expansion \eqref{estim} are:
$$
\dfrac{1}{\sqrt{4\pi \beta_1 n}} \exp \left( -\dfrac{(\ell-\alpha_1 n)^2}{4 \beta_1 n} \right) 
-\dfrac{\gamma_{1,3}}{3\, ! \, n} \, \left(H^{\beta_1}_2\right)'''\left( \dfrac{\ell-\alpha_1 n}{\sqrt{n}} \right) \, ,
$$
and these two first terms can be rewritten as (recall the relation $2 \, \beta_1=\sigma_1^2$):
$$
\dfrac{1}{\sqrt{2 \, \pi \sigma_1^2 n}} \exp \left( -\dfrac{x^2}{2} \right) 
-\dfrac{\gamma_{1,3}}{3\, ! \, \sigma_1^4 \, n} \, (x^3-3 \, x) \, \dfrac{1}{\sqrt{2 \, \pi}} \, \exp \left( -\dfrac{x^2}{2} \right) \, ,\quad 
\text{\rm with } \quad x:=\dfrac{\ell-\alpha_1\, n}{\sigma_1 \, \sqrt{n}} \, ,
$$
which coincides with the expression in \cite[Theorem 13, page 205]{Petrov}. If we compare Theorem \ref{thm1} above with \cite[Chapter VII]{Petrov}, the novelty is the (sharp) estimate of the remainder with a Gaussian bound rather than a uniform bound as $o(n^{-(M+1)/2})$ or $O(n^{-(M+2)/2})$ which would not allow to obtain error bounds such as \eqref{estimcoro1} in Corollary \ref{coro1} below. The price to pay is the holomorphy condition on $F_{\boldsymbol{a}}$ in Assumption \ref{ass:1}.

One can also try to compare Theorem \ref{thm1} with the well-known stability results derived in \cite{Thomee}, see also \cite[Chapter 5]{BTW}, for finite difference schemes. Actually, the article \cite{Thomee} is the fundamental reference where the condition \eqref{DLFa} on the Fourier transform $F_a$ 
was first highlighted. The emphasis in \cite{Thomee} or \cite{BTW} was rather on deriving stability bounds on the iterated convolution operator ${\boldsymbol{u}} \mapsto {\boldsymbol{a}}^{\star n} \star {\boldsymbol{u}}$. These bounds are used in \cite{BTW} to derive convergence estimates, mostly on finite time intervals. Our goal here is slightly different since we rather aim at studying the large time behavior of the iterated convolution operator which could open the possibility of studying large time convergence estimates.

An immediate corollary of Theorem \ref{thm1} is the following result which gives an accurate description of the large time behavior of the iterated convolution ${\boldsymbol{a}}^{\star n} \star {\boldsymbol{u}}^0$ for any initial condition ${\boldsymbol{u}}^0 \in \ell^p (\Z;\C)$. Corollary \ref{coro1} below yields an explicit expression for the large time asymptotics of a finite difference scheme corresponding to the convolution with a sequence ${\boldsymbol{a}}$ that satisfies both Assumptions \ref{ass:1} and \ref{ass:2}. It can be used to justify in a sharp quantitative way the link between a numerical scheme and its associated \emph{modified equation}.

\begin{corollary}
\label{coro1}
Let the sequence ${\boldsymbol{a}}$ satisfy Assumptions \ref{ass:1} and \ref{ass:2}. Then for any integer $M \in \N$, there exists a constant $C_M>0$ such that for any sequence ${\boldsymbol{u}}^0 \in \ell^p (\Z;\C)$ with $1 \le p \le +\infty$ and for any $n \in \N^*$, there holds:
\begin{multline}
\label{estimcoro1}
\left\| {\boldsymbol{a}}^{\star n} \star {\boldsymbol{u}}^0 -\sum_{k=1}^K \dfrac{\underline{\kappa}_k^{-\ell} F_{\boldsymbol{a}}(\underline{\kappa}_k)^n}{n^{1/(2\, \mu_k)}} \, H^{\beta_k}_{2 \mu_k} \left( \dfrac{\cdot-\alpha_k n}{n^{1/(2\, \mu_k)}} \right) \star {\boldsymbol{u}}^0 \right. \\
\left. -\sum_{m=1}^M \sum_{k=1}^K 
\dfrac{\underline{\kappa}_k^{-\ell} F_{\boldsymbol{a}}(\underline{\kappa}_k)^n}{n^{(m+1)/(2\, \mu_k)}} \, \Big( P_{k,m} (-{\rm d}/{\rm d}x) H^{\beta_k}_{2 \mu_k} \Big) \left( \dfrac{\cdot-\alpha_k n}{n^{1/(2\, \mu_k)}} \right) \star {\boldsymbol{u}}^0 \right\|_{\ell^p} \le 
\dfrac{C_M \, \| {\boldsymbol{u}}^0 \|_{\ell^p}}{n^{(M+1)/(2\, \mu)}} \, ,
\end{multline}
with $\mu:=\max_k \mu_k$.
\end{corollary}

Some analogues of Theorem \ref{thm1} are proved in \cite{RSC1,RSC2,Coeuret,Randles}, sometimes with restrictions on $K$ (number of tangency points) and/or on the drifts $\alpha_k$, and/or on $M$ (number of terms in the asymptotic expansion). As far as we know, our framework seems to be the most general so far. Our only (crucial !) assumption is the fact that $F_{\boldsymbol{a}}$ has an expansion of the form \eqref{DLFa} at every tangency point. According to the main result in \cite{Thomee}, this corresponds to a \emph{stable} situation where the geometric sequence $({\boldsymbol{a}}^{\star n})_{n \in \N^*}$ is bounded in $\ell^1(\Z;\C)$. From the point of view of partial differential equations, the above asymptotic expansion involves the fundamental solution of some parabolic equation.

Other forms for the expansion \eqref{DLFa} would lead to dispersive behaviors, a prototype of which (the so-called Lax-Wendroff scheme, see \cite{BW}) is studied in depth in \cite{jfcAMBP}. In the dispersive case, the geometric sequence $({\boldsymbol{a}}^{\star n})_{n \in \N^*}$ is no longer bounded in $\ell^1(\Z;\C)$. The reference \cite[Chapter 5]{BW} provides with several comparisons for the Green's function in this case but a complete derivation of a large time expansion with sharp bounds is still missing. We shall deal with the local limit theorem (up to any order) in the dispersive case in a subsequent work. Another very interesting perspective would be to derive analogues of Theorem \ref{thm1} in the multi-dimensional setting where the sequence $\boldsymbol{a}$ is now indexed on the $d$-dimensional lattice $\mathbb{Z}^d$, with $d\geq2$. We refer to \cite{bui,Randles,RSC2,RSC3} for recent developments on the subject. Let us emphasize that the multi-dimensional setting presents a much richer classification of the tangency points compared to the one-dimensional case which has only two cases\footnote{The two cases are the parabolic case, as studied here, and the dispersive case, as addressed in \cite{Thomee,hedstrom,BTW,jfcAMBP,RSC1}.}. Borrowing the terminology used in the aforementioned references, the multi-dimensional analogue of the parabolic case studied here would correspond to the case where all tangency points of $F_{\boldsymbol{a}}$ are of positive homogeneous type, see \cite[Definition 1.3.]{RSC2}.

When the sequence ${\boldsymbol{a}}$ has finite support, the bound \eqref{estim} holds true not only in a large sector $\{ |\ell| \le L \, n\}$ but for any $\ell \in \Z$ (see \cite{Coeuret}). The statement of Theorem \ref{thm1} in that case is therefore more simple since one obtains a global bound for the error between ${\boldsymbol{a}}^{\star n}_\ell$ and its asymptotic expansion. In the general case of a holomorphic $F_{\boldsymbol{a}}$, the decay of ${\boldsymbol{a}}^{\star n}_\ell$ at infinity (for a fixed $n$) is at best exponential which makes the distinction of the two regimes in Theorem \ref{thm1} relevant. Several applications of Theorem \ref{thm1} are given below in Section \ref{section4}. We now give the proof of Theorem \ref{thm1}.

\section{Proof of Theorem \ref{thm1}}
\label{section3}

Theorem \ref{thm1} will ultimately justify that the iterated convolution ${\boldsymbol{a}}^{\star n}_\ell$ is mostly concentrated around the $K$ sectors 
$\cup_k \{ (\ell,n) \, | \, \ell \sim \alpha_k \, n\}$, meaning in particular that ${\boldsymbol{a}}^{\star n}_\ell$ is to some extent negligible for $|\ell | \gg n$. 
We therefore first study this far field regime where $(\ell,n)$ is outside a large enough sector and then turn to the main purpose of this work which is 
the proof of \eqref{estim}.

\subsection{The far field regime}

We first consider the case where the ratio $|\ell|/n$ is large. We start from the expression:
$$
{\boldsymbol{a}}^{\star n}_\ell =\dfrac{1}{2 \, \pi} \int_{-\pi}^\pi {\rm e}^{-\mathbf{i} \,  \ell \, \theta} \, F_{\boldsymbol{a}}({\rm e}^{\mathbf{i} \, \theta})^n \, {\rm d}\theta \, ,
$$
and use the fact that the function $F_{\boldsymbol{a}}$ has a holomorphic extension on an annulus around $\cercle$. Let us assume $\ell>0$ and let us change the integration contour by using the Cauchy formula. Choosing for instance $\delta:=\ln (1+\varepsilon/2)$, with $\varepsilon>0$ given by Assumption \ref{ass:1}, we obtain:
$$
{\boldsymbol{a}}^{\star n}_\ell =\dfrac{1}{2 \, \pi} \int_{-\pi}^\pi {\rm e}^{-\ell \, \delta} {\rm e}^{-\mathbf{i} \, \ell \, \theta} \, F_{\boldsymbol{a}}({\rm e}^{\mathbf{i} \, (\theta-\mathbf{i} \, \delta)})^n \, {\rm d}\theta \, .
$$
Applying the triangle inequality, we thus get the bound:
$$
|{\boldsymbol{a}}^{\star n}_\ell| \le {\rm e}^{-\ell \, \delta} \, \left( \sup_{\theta \in [-\pi,\pi]} |F_{\boldsymbol{a}}({\rm e}^\delta \, {\rm e}^{\mathbf{i} \, \theta})| \right)^n \, ,
$$
meaning that for some well-chosen constant $C_\flat>0$ that does not depend on $\ell$ and $n$, there holds:
$$
|{\boldsymbol{a}}^{\star n}_\ell| \le \exp \Big( -\delta \, \ell +C_\flat \, n \Big) \, ,
$$
as long as $n \in \N^*$ and $\ell \ge 0$. For $\ell \ge (2\, C_\flat/\delta) \, n$, we thus have:
$$
|{\boldsymbol{a}}^{\star n}_\ell| \le \exp \left( -\dfrac{\delta}{2} \, \ell \right) 
\le \exp \left( -\dfrac{\delta}{4} \, \ell  -\dfrac{C_\flat}{2} \, n \right) \, .
$$
We have thus proved the validity of \eqref{estim1} for $\ell \ge L \, n$ and $L>0$ is a well-chosen constant (there is no loss of generality in assuming that $L$ is an integer up to increasing $L$). The case $\ell<0$ is entirely similar and is left to the interested reader.

From now on, the integer $L$ of Theorem \ref{thm1} is fixed as above, and we shall always assume that $\ell \in \Z$ is restricted to those values for which $|\ell| \le L \, n$, with $n \in \N^*$ the discrete time.

\subsection{The local limit theorem}

\subsubsection{Fixing the constants}

The aim of this paragraph is to fix several constants that will arise in various estimates for proving Theorem \ref{thm1}. We start with the following simple observation.

\begin{lemma}
\label{lem:beta}
Let $\beta \in \C$ have positive real part and let $\mu \in \N^*$. Then there exists a constant $C>0$ such that for any $u \in \C$, there holds:
$$
\Re \, (\beta \, u^{2\, \mu}) \ge \dfrac{\Re \, \beta}{2} \, (\Re \, u)^{2\, \mu} - C \, (\Im \, u)^{2\, \mu} \, .
$$
\end{lemma}

\begin{proof}
The proof simply consists in expanding:
$$
\Re \, (\beta \, u^{2\, \mu})=(\Re \, \beta) \, (\Re \, u^{2\, \mu}) 
-(\Im \, \beta) \, (\Im \, u^{2\, \mu}) \, ,
$$
and then in expanding both the real and imaginary parts of $u^{2\, \mu}$. 
Writing $u=x+\mathbf{i}\, y$, we obtain:
\begin{align*}
\Re \, (\beta \, u^{2\, \mu})=(\Re \, \beta) \, x^{2\, \mu} 
&+(\Re \, \beta) \, \sum_{m=1}^\mu \binom{2 \, \mu}{2 \, m} (-1)^m \, x^{2(\mu-m)} \, y^{2 \, m} \\
&-(\Im \, \beta) \, \sum_{m=0}^{\mu-1} \binom{2 \, \mu}{2 \, m+1} (-1)^m \, x^{2(\mu-m)-1} \, y^{2 \, m+1} \, .
\end{align*}
By repeatedly using Holder's and Young's inequalities, we can obtain the estimate:
\begin{multline*}
\left| (\Re \, \beta) \, \sum_{m=1}^\mu \binom{2 \, \mu}{2 \, m} (-1)^m \, x^{2(\mu-m)} \, y^{2 \, m} -(\Im \, \beta) \, \sum_{m=0}^{\mu-1} \binom{2 \, \mu}{2 \, m+1} (-1)^m \, x^{2(\mu-m)-1} \, y^{2 \, m+1} \right| \\
\le \dfrac{\Re \, \beta}{2} \, x^{2\, \mu} +C \, y^{2\, \mu} \, ,
\end{multline*}
with a large enough constant $C$ (that does not depend on $x$ and $y$). This gives the result of Lemma \ref{lem:beta}.
\end{proof}

Using Assumption \ref{ass:2} and Lemma \ref{lem:beta}, we can fix once and for all two constants $\beta_*>0$ and $\beta^* \ge \beta_*$ such that the following inequalities hold true:
\begin{equation}
\label{defbetaetoiles}
\forall \, k=1,\dots,K \, ,\quad \forall \, u \in \C \, ,\quad 
\Re \, (\beta_k \, u^{2\, \mu_k}) \ge \beta_* \, (\Re \, u)^{2\, \mu_k} - \beta^* \, (\Im \, u)^{2\, \mu_k} \, .
\end{equation}
These inequalities will be helpful later on. We now prove another useful preliminary result.

\begin{lemma}
\label{lem:reste}
Let $k \in \{ 1,\dots,K \}$, and let the coefficients $\gamma_{k,\nu}$ be defined in \eqref{DLk} for $\nu \ge 2\, \mu_k+1$. Then there exists $\delta_0>0$ such that the function:
\begin{equation}
\label{defgk}
g_k \, : \, (w,z) \in \C \times B(0,\delta_0) \longmapsto 
\exp \left( w^{2\, \mu_k} \, \sum_{\nu \ge 1} \dfrac{\gamma_{k,2\, \mu_k+\nu}}{(2\, \mu_k+\nu) \, !} \, z^\nu  \right) \, ,
\end{equation}
is holomorphic on $\C \times B(0,\delta_0)$. Furthermore, for any $M \in \N$, there exists a constant $C=C(k,M)>0$ and some $\delta \in (0,\delta_0)$ such that there holds:
$$
\forall \, (w,z) \in \C \times \overline{\mathcal{C}(0,\delta)} \, , \, 
\left| g_k(w,z) -\sum_{m=0}^M \dfrac{\partial^m g_k}{\partial z^m}(w,0) \, \dfrac{z^m}{m \, !}  \right| \le C \, |z|^{M+1} \, \exp \left( \dfrac{\beta_*}{2} \, (\Re \, w)^{2\, \mu_k} + \beta^* \, (\Im \, w)^{2\, \mu_k}\right) ,
$$
where $\beta_*$ and $\beta^*$ are the constants in \eqref{defbetaetoiles}.
\end{lemma}

\begin{proof}
Since $F_{\boldsymbol{a}}$ is holomorphic on an annulus $\{ \zeta \in \C \, | \, 1-\varepsilon < |\zeta| <1+\varepsilon \}$, for some $\varepsilon>0$, the power series (in $Z$):
$$
\sum_{\nu \ge 2\, \mu_k+1} \dfrac{\gamma_{k,\nu}}{\nu \, !} \, Z^\nu \, ,
$$
has a positive convergence radius. This implies that for some $\delta_0>0$, the function $g_k$ defined in \eqref{defgk} is holomorphic on $\C \times B(0,\delta_0)$. From now on, the parameter $\delta_0$ is fixed and we consider $\delta \in (0,\delta_0/2)$ in such a way that the closed square $\overline{\mathcal{C}(0,\delta)}$ is included in the open ball $B(0,\delta_0)$. This will allow us to obtain some uniform bounds with respect to the radius $\delta$.

Applying the Taylor formula, we get:
\begin{equation}
\label{taylorgk}
g_k(w,z) -\sum_{m=0}^M \dfrac{\partial^m g_k}{\partial z^m}(w,0) \, \dfrac{z^m}{m \, !} 
\, = \, \dfrac{z^{M+1}}{M \, !} \, \int_0^1 (1-t)^M \, \dfrac{\partial^{M+1} g_k}{\partial z^{M+1}}(w,t\, z) \, {\rm d}t \, .
\end{equation}
The $(M+1)$-th partial derivative of $g_k$ with respect to $z$ is computed by using the Fa\`a di Bruno formula \cite{Comtet}. We introduce the notation:
$$
r_k(z) \, := \, \sum_{\nu \ge 1} \dfrac{\gamma_{k,2\, \mu_k+\nu}}{(2\, \mu_k+\nu) \, !} \, z^\nu \, ,
$$
so that the function $g_k$ reads:
$$
\forall \, (w,z) \in \C \times B(0,\delta_0) \, ,\quad g_k(w,z)=\exp \big( w^{2 \, \mu_k} \, r_k(z) \big) \, .
$$
We now consider $(w,z) \in \C \times \overline{\mathcal{C}(0,\delta)}$ and compute:
$$
\dfrac{1}{(M+1)\, !} \, \dfrac{\partial^{M+1} g_k}{\partial z^{M+1}}(w,z) 
= \exp \big( w^{2 \, \mu_k} \, r_k(z) \big) \, \sum_{\langle \boldsymbol{\nu} \rangle =M+1} \dfrac{w^{2 \, \mu_k \, |\boldsymbol{\nu}|}}{\boldsymbol{\nu} \, !} \, \prod_{\ell \ge 1} \left( \dfrac{r_k^{(\ell)}(z)}{\ell \, !} \right)^{\nu_\ell} \, ,
$$
where the notation $\boldsymbol{\nu}$ refers to a finitely supported integer valued sequence $(\nu_1,\nu_2,\dots)$, and we use the notation\footnote{All these quantities make sense for finitely supported sequences as we consider here.}:
$$
\langle \boldsymbol{\nu} \rangle := \sum_{\ell \ge 1} \, \ell \, \nu_\ell \, ,\quad 
|\boldsymbol{\nu}| := \sum_{\ell \ge 1} \, \nu_\ell \, ,\quad 
\boldsymbol{\nu} \, ! := \prod_{\ell \ge 1} \, \nu_\ell \, ! \, .
$$
We now use uniform bounds for the derivatives $r_k^{(\ell)}$ on the closed square $\overline{\mathcal{C}(0,\delta_0/2)}$, and we use the bound:
$$
\sup_{z \in \overline{\mathcal{C}(0,\delta)}} \, |r_k(z)| \le C_\sharp \, \delta \, ,
$$
where the constant $C_\sharp$ is independent of $\delta \in (0,\delta_0/2)$. The crucial fact here is that $r_k$ vanishes at $0$ so the uniform bound for $r_k$ is at least linear with respect to $\delta$. We end up with:
$$
\forall \, (w,z) \in \C \times \overline{\mathcal{C}(0,\delta)} \, ,\quad  \left| \, \dfrac{\partial^{M+1} g_k}{\partial z^{M+1}}(w,z) \, \right| 
\le Q(|w|)\exp \big( C_\sharp \, \delta \, |w|^{2 \, \mu_k} \big) \, ,
$$
where $Q$ is some real polynomial with nonnegative coefficients (that depend on $M$). Up to choosing $\delta$ small enough, we thus get the bound:
$$
\forall \, (w,z) \in \C \times \overline{\mathcal{C}(0,\delta)} \, ,\quad \left| \, \dfrac{\partial^{M+1} g_k}{\partial z^{M+1}}(w,z) \, \right| 
\le C_M \, \exp \left( \dfrac{\beta_*}{2} \, (\Re \, w)^{2\, \mu_k} + \beta^* \, (\Im \, w)^{2\, \mu_k} \right) \, ,
$$
for some suitable constant $C_M>0$. Using this bound in the Taylor formula \eqref{taylorgk} for $g_k$, we eventually get the result of Lemma \ref{lem:reste}.
\end{proof}

From now on, we consider a given integer $M \in \N$. With the help of Lemma \ref{lem:reste}, we fix once and for all a radius $\delta>0$ and a constant $C>0$ such that, for any $k \in \{ 1,\dots,K \}$, the function $g_k$ in \eqref{defgk} is holomorphic on $\C \times B(0,2\, \delta)$ and satisfies the following bound:
\begin{equation}
\label{estimgk}
\forall \, (w,z) \in \C \times \overline{\mathcal{C}(0,\delta)} \, ,\quad 
\left| g_k(w,z) -\sum_{m=0}^M \dfrac{\partial^m g_k}{\partial z^m}(w,0) \, \dfrac{z^m}{m \, !} \right| 
\le C \, |z|^{M+1} \exp \left( \dfrac{\beta_*}{2} \, (\Re \, w)^{2\, \mu_k} + \beta^* \, (\Im \, w)^{2\, \mu_k}\right) ,
\end{equation}
where $\beta_*$ and $\beta^*$ are the same constants as in \eqref{defbetaetoiles}. There is no loss of generality in assuming that the $K$ subsets $\{ \underline{\kappa}_k \, {\rm e}^z \, | \, z \in \overline{\mathcal{C}(0,\delta)} \}$ do not intersect one another (this amounts to choosing $\delta$ small enough since the points $\underline{\kappa}_k$ are pairwise distinct).
\bigskip

The final ingredient for proving Theorem \ref{thm1} is the following observation.

\begin{lemma}
\label{lem:expressionPkm}
Let $k \in \{ 1,\dots,K \}$, and let the polynomial $P_{k,m}$ be defined by the asymptotic 
expansion \eqref{defPknu} for any $m \ge 1$. Let also the function $g_k$ be defined by \eqref{defgk}. Then there holds:
$$
\forall \, m \ge 1 \, ,\quad \forall \, w \in \C \, ,\quad P_{k,m}(w) =\dfrac{w^m}{m \, !} \, \dfrac{\partial^m g_k}{\partial z^m}(w,0) \, .
$$
\end{lemma}

\begin{proof}
The proof is rather straightforward since, by the definition \eqref{defPknu}, we have\footnote{We feel free to skip the justification about the holomorphy of the considered functions on appropriate domains of $\C^2$.}:
\begin{align*}
P_{k,m}(w) &=\dfrac{1}{m \, !} \, \dfrac{\partial^m}{\partial Z^m} \left( \exp \left( \sum_{\nu \ge 1} \dfrac{\gamma_{k,2\, \mu_k+\nu}}{(2\, \mu_k+\nu) \, !} \, w^{2\, \mu_k+\nu} \, Z^\nu \right) \right) \Big|_{Z=0} \\
&=\dfrac{1}{m \, !} \, \dfrac{\partial^m}{\partial Z^m} \Big( g_k(w,w\, Z) \Big) \Big|_{Z=0} = \dfrac{w^m}{m \, !} \, \dfrac{\partial^m g_k}{\partial z^m}(w,0) \, .
\end{align*}
\end{proof}

Combining Lemma \ref{lem:expressionPkm} and the Fa\`a di Bruno formula, we obtain the same  expression as in \cite[Chapter VII]{Petrov}, that is:
\begin{equation}
\label{valeurPkm}
P_{k,m}(X) =X^m \, \sum_{\langle \boldsymbol{\nu} \rangle =m} \dfrac{X^{2 \, \mu_k \, |\boldsymbol{\nu}|}}{\boldsymbol{\nu} \, !} \, \prod_{\ell \ge 1} \left( \dfrac{\gamma_{k,2\, \mu_k+\ell}}{(2 \, \mu_k+\ell) \, !} \right)^{\nu_\ell} \, .
\end{equation}
For instance, we have for any $k \in \{ 1,\dots,K \}$:
$$
P_{k,1}(X)=\dfrac{\gamma_{k,2\, \mu_k+1}}{(2\, \mu_k+1) \, !} \, X^{2\, \mu_k+1} \, ,
$$
and in the case $\gamma_{k,2\, \mu_k+1}=0$, we have furthermore (see Section \ref{section4} for an example):
$$
P_{k,2}(X)=\dfrac{\gamma_{k,2\, \mu_k+2}}{(2\, \mu_k+2) \, !} \, X^{2\, \mu_k+2} \, ,
$$
We now turn to the proof of Theorem \ref{thm1}.

\subsubsection{Proof of Theorem \ref{thm1}}

We warn the reader that many constants appear below. From now on, large positive constants are always denoted $C$ and small positive constants are denoted $c$, with the convention that constants may be relabelled from one line to the other or within the same line. Constants may depend on the integer $M$, that is given, but are always independent of $n$ and $\ell$.

Based on the various constants that have been fixed in the previous paragraph, let us choose some real number $\underline{\theta}$ such that $\exp(\mathbf{i} \, \underline{\theta})$ does not belong to any of the arcs $\{ \underline{\kappa}_k {\rm e}^{\mathbf{i} \, \theta} \, | \, \theta \in [-\delta,\delta] \}$ of the unit circle (this is possible because these arcs do not intersect one another). We start from the expression:
\begin{equation}
\label{integrale1}
{\boldsymbol{a}}^{\star n}_\ell =\dfrac{1}{2 \, \pi} \int_{\underline{\theta}}^{\underline{\theta}+2\, \pi} {\rm e}^{-\mathbf{i} \, \ell \, \theta} \, F_{\boldsymbol{a}}({\rm e}^{\mathbf{i} \, \theta})^n \, {\rm d}\theta \, .
\end{equation}
We first split the integral in \eqref{integrale1}. For any $k=1,\dots,K$, we consider $\theta_k \in \R$ such that $\underline{\kappa}_k=\exp(\mathbf{i} \, \theta_k)$ and $\theta_k$ belongs to the open interval $(\underline{\theta},\underline{\theta}+2\, \pi)$. Then the intervals $[\theta_k-\delta,\theta_k+\delta]$ are pairwise disjoint and each of them is included in $(\underline{\theta},\underline{\theta}+2\, \pi)$. The important remark is that when $\theta$ does not belong to a segment of the form $[\theta_k-\delta,\theta_k+\delta]$, then the factor $F_{\boldsymbol{a}}({\rm e}^{\mathbf{i} \theta})$ in \eqref{integrale1} belongs to the open unit disk. Splitting the integral in \eqref{integrale1} and using a continuity argument, there exists a positive constant $c$ such that:
\begin{equation}
\label{integrale2}
\left| {\boldsymbol{a}}^{\star n}_\ell -\sum_{k=1}^K \dfrac{1}{2 \, \pi} \int_{\theta_k-\delta}^{\theta_k+\delta} {\rm e}^{-\mathbf{i} \, \ell \, \theta} \, F_{\boldsymbol{a}}({\rm e}^{\mathbf{i} \, \theta})^n \, {\rm d}\theta \right| \le {\rm e}^{-c \, n} \, ,
\end{equation}
uniformly with respect to $\ell$ and $n$. Performing changes of variables in each integral on the left hand side of \eqref{integrale2} (just shift the interval in $\theta$), we thus get:
\begin{equation}
\label{integrale3}
\left| {\boldsymbol{a}}^{\star n}_\ell -\sum_{k=1}^K \dfrac{\underline{\kappa}_k^{-\ell} F_{\boldsymbol{a}}(\underline{\kappa}_k)^n}{2 \, \pi} \int_{-\delta}^{+\delta} {\rm e}^{-\mathbf{i} \, \ell \, \theta} \, \Big( F_{\boldsymbol{a}}(\underline{\kappa}_k)^{-1} \, F_{\boldsymbol{a}}(\underline{\kappa}_k \, {\rm e}^{\mathbf{i} \, \theta}) \Big)^n \, {\rm d}\theta \right| \le {\rm e}^{-c \, n} \, .
\end{equation}
We now use the holomorphy of $\theta \mapsto F_{\boldsymbol{a}}(\underline{\kappa}_k \, {\rm e}^{\mathbf{i} \, \theta})$ on $\overline{\mathcal{C}(0,\delta)}$ together with the convergent power series expansion \eqref{DLk}. This modifies \eqref{integrale3} accordingly into:
\begin{equation}
\label{integrale4}
\left| {\boldsymbol{a}}^{\star n}_\ell -\sum_{k=1}^K \dfrac{\underline{\kappa}_k^{-\ell} F_{\boldsymbol{a}}(\underline{\kappa}_k)^n}{2 \, \pi} \int_{-\delta}^{+\delta} {\rm e}^{-\mathbf{i} (\ell-n\alpha_k) \theta} \, {\rm e}^{-n \beta_k \theta^{2\, \mu_k}} \exp \left( n \, \sum_{\nu \ge 2\, \mu_k+1} \dfrac{\gamma_{k,\nu}}{\nu \, !} (\mathbf{i} \theta)^\nu \right) \, {\rm d}\theta \right| \le {\rm e}^{-c \, n} \, .
\end{equation}
For ease of notation, we introduce the notation:
\begin{equation}
\label{defxk}
\forall \, k=1,\dots,K \, ,\quad x_k := \dfrac{\ell-n\alpha_k}{n^{1/(2\, \mu_k)}} \, .
\end{equation}
We change variables $\theta \rightarrow \theta/n^{1/(2\, \mu_k)}$ in each integral on the left hand side of \eqref{integrale4} to get:
\begin{equation}
\label{estimation1}
\left| {\boldsymbol{a}}^{\star n}_\ell -\sum_{k=1}^K \dfrac{\underline{\kappa}_k^{-\ell} F_{\boldsymbol{a}}(\underline{\kappa}_k)^n}{2 \, \pi \, n^{1/(2\, \mu_k)}} \int_{-\delta n^{1/(2\, \mu_k)}}^{+\delta n^{1/(2\, \mu_k)}} {\rm e}^{-\mathbf{i} \, x_k \, \theta} \, {\rm e}^{-\beta_k \, \theta^{2\, \mu_k}} g_k \left( \mathbf{i} \, \theta,\dfrac{\mathbf{i} \, \theta}{n^{1/(2\, \mu_k)}} \right) \, {\rm d}\theta \right| \le {\rm e}^{-c \, n} \, ,
\end{equation}
where we have used the notation $g_k$ introduced in Lemma \ref{lem:reste}.
\bigskip

The final step of the proof consists in approximating $g_k$ by its Taylor expansion with respect to its second argument (this is where Lemma \ref{lem:reste} will be useful) and then by approximating the integral on the large segment $[-\delta n^{1/(2\, \mu_k)},+\delta n^{1/(2\, \mu_k)}]$ by the integral over $\R$. Namely, by using the triangle inequality, the preliminary estimate \eqref{estimation1} yields\footnote{Recall that the integer $M$ has been given since the previous paragraph.}:
\begin{equation}
\label{estimation2}
\left| {\boldsymbol{a}}^{\star n}_\ell -\sum_{k=1}^K \dfrac{\underline{\kappa}_k^{-\ell} F_{\boldsymbol{a}}(\underline{\kappa}_k)^n}{2 \, \pi \, n^{1/(2\, \mu_k)}} \int_\R {\rm e}^{-\mathbf{i} \, x_k \, \theta} \, {\rm e}^{-\beta_k \, \theta^{2\, \mu_k}} \sum_{m=0}^M \dfrac{(\mathbf{i} \, \theta)^m}{m\, ! \, n^{m/(2\, \mu_k)}} \, \dfrac{\partial^m g_k}{\partial z^m} (\mathbf{i} \, \theta,0) \, {\rm d}\theta \right| \le {\rm e}^{-c \, n} 
+|\varepsilon_{\ell,n}^1| +|\varepsilon_{\ell,n}^2| \, ,
\end{equation}
where the error terms $\varepsilon_{\ell,n}^1$ and $\varepsilon_{\ell,n}^2$ are defined as follows:
\begin{align}
\varepsilon_{\ell,n}^1 &:= \sum_{k=1}^K \dfrac{\underline{\kappa}_k^{-\ell} F_{\boldsymbol{a}}(\underline{\kappa}_k)^n}{2 \, \pi \, n^{1/(2\, \mu_k)}} \int_{-\delta \,  n^{1/(2\, \mu_k)}}^{+\delta \, n^{1/(2\, \mu_k)}} {\rm e}^{-\mathbf{i} \, x_k \, \theta}  {\rm e}^{-\beta_k \, \theta^{2\, \mu_k}} \notag \\
& \qquad \qquad \qquad \qquad \qquad \qquad \qquad \times \left( g_k \left( \mathbf{i} \, \theta,\dfrac{\mathbf{i} \, \theta}{n^{1/(2\, \mu_k)}} \right) -\sum_{m=0}^M \dfrac{(\mathbf{i} \, \theta)^m}{m\, ! \, n^{m/(2\, \mu_k)}} \, \dfrac{\partial^m g_k}{\partial z^m} (\mathbf{i} \, \theta,0) \right) {\rm d}\theta ,\label{deferror1} \\
\varepsilon_{\ell,n}^2 &:= \sum_{k=1}^K \dfrac{\underline{\kappa}_k^{-\ell} F_{\boldsymbol{a}}(\underline{\kappa}_k)^n}{2 \, \pi \, n^{1/(2\, \mu_k)}} \int_{\R \setminus [-\delta \, n^{1/(2\, \mu_k)},+\delta \, n^{1/(2\, \mu_k)}]}  {\rm e}^{-\mathbf{i} \, x_k \, \theta} \, {\rm e}^{-\beta_k \, \theta^{2\, \mu_k}} \sum_{m=0}^M \dfrac{(\mathbf{i} \, \theta)^m}{m\, ! \, n^{m/(2\, \mu_k)}} \, \dfrac{\partial^m g_k}{\partial z^m} (\mathbf{i} \, \theta,0) \, {\rm d}\theta \, .\label{deferror2}
\end{align}

\paragraph{Estimate of the error terms.} It is useful below to adopt the convention $P_{k,0}(X):=1$ so that the result of Lemma \ref{lem:expressionPkm} holds not only for $m \in \N^*$ but for any integer $m$.

Let us start with the estimate of the error term $\varepsilon_{\ell,n}^2$ defined in \eqref{deferror2}. We recall that both $\underline{\kappa}_k$ and $F_{\boldsymbol{a}}(\underline{\kappa}_k)$ have modulus $1$, and we apply Lemma \ref{lem:expressionPkm} to simplify the integrand in \eqref{deferror2}. We obtain:
$$
|\varepsilon_{\ell,n}^2| \le \sum_{k=1}^K \sum_{m=0}^M \dfrac{1}{2 \, \pi \,  n^{(m+1)/(2\, \mu_k)}} \int_{\R \setminus [-\delta \, n^{1/(2\, \mu_k)},+\delta \, n^{1/(2\, \mu_k)}]} {\rm e}^{-(\text{\rm Re} \, \beta_k) \, \theta^{2\, \mu_k}} |P_{k,m}(\mathbf{i} \, \theta)| \, {\rm d}\theta \, .
$$
Recalling that all the $\beta_k$'s have positive real part and $n \ge 1$, we find that there exists a constant $C$ that is independent of $n$ and $\ell$ such that there holds:
$$
|\varepsilon_{\ell,n}^2| \le C \, \sum_{k=1}^K \int_{\R \setminus [-\delta \, n^{1/(2\, \mu_k)},+\delta \, n^{1/(2\, \mu_k)}]} {\rm e}^{-(\text{\rm Re} \, \beta_k/2) \, \theta^{2\, \mu_k}} \, {\rm d}\theta \, .
$$
We thus obtain the exponential bound:
\begin{equation}
\label{estimerror2}
|\varepsilon_{\ell,n}^2| \le C \, {\rm e}^{-c \, n} \, ,
\end{equation}
with positive constants $C$ and $c$ that do not depend on $\ell$ and $n$.
\bigskip

It remains to obtain a bound for the first error term $\varepsilon_{\ell,n}^1$ defined in \eqref{deferror1}. There holds:
\begin{multline}
\label{estimerror1-init}
|\varepsilon_{\ell,n}^1| \le \sum_{k=1}^K n^{-1/(2\, \mu_k)} \left| \int_{-\delta \, n^{1/(2\, \mu_k)}}^{+\delta \, n^{1/(2\, \mu_k)}} {\rm e}^{-\mathbf{i} \, x_k \, \theta} \, {\rm e}^{-\beta_k \, \theta^{2\, \mu_k}} \right. \\
\left. \left( g_k \left( \mathbf{i} \, \theta,\dfrac{\mathbf{i} \, \theta}{n^{1/(2\, \mu_k)}} \right) -\sum_{m=0}^M \dfrac{(\mathbf{i} \, \theta)^m}{m\, ! \, n^{m/(2\, \mu_k)}} \, \dfrac{\partial^m g_k}{\partial z^m} (\mathbf{i} \, \theta,0) \right) {\rm d}\theta \right| ,
\end{multline}
and we are now going to use a contour deformation (for each $k$) in order to derive a sharp bound for $\varepsilon_{\ell,n}^1$ (this is the reason why we have not used the triangle inequality in the integrals so far).

We consider an integer $k \in \{ 1,\dots,K \}$ and assume that $x_k$ in \eqref{defxk} is nonnegative (similar arguments yield analogous bounds when $x_k$ is negative, the contour depicted in Figure \ref{fig:contour1} below has just to be switched to the upper half complex plane). We then define:
\begin{equation}
\label{defXi}
\Xi := \begin{cases}
\left( \dfrac{x_k}{4 \, \mu_k \, \beta^*} \right)^{1/(2\, \mu_k-1)} \, ,& 
\text{\rm if } \dfrac{x_k}{4 \, \mu_k \, \beta^*} \le \delta^{2\, \mu_k-1} \, n^{(2\, \mu_k-1)/(2\, \mu_k)} \, ,\\
\, & \\
\delta \, n^{1/(2\, \mu_k)} \, ,&\text{\rm if } \dfrac{x_k}{4 \, \mu_k \, \beta^*} \ge \delta^{2\, \mu_k-1} \, n^{(2\, \mu_k-1)/(2\, \mu_k)} \, ,
\end{cases}
\end{equation}
so that, in particular, there always holds:
\begin{equation}
\label{estimsegment1}
\forall \, u \in [0,\Xi] \, ,\quad 0 \le 2 \, \beta^* \, u^{2\, \mu_k-1} \le \dfrac{x_k}{2\, \mu_k} \, ,
\end{equation}
and $\Xi/n^{1/(2\, \mu_k)} \le \delta$. Consequently, for any $z$ on the contour that is depicted in blue in Figure \ref{fig:contour1}, we have $\max (|\text{\rm Re } z|,|\text{\rm Im } z|)/n^{1/(2\, \mu_k)} \le \delta$ and we shall therefore be able to apply Cauchy's formula for holomorphic functions and also use the estimate of Lemma \ref{lem:reste}.

Namely, Cauchy's formula gives:
\begin{multline*}
\int_{-\delta \, n^{1/(2\, \mu_k)}}^{+\delta \, n^{1/(2\, \mu_k)}} {\rm e}^{-\mathbf{i} \, x_k \, \theta} \, {\rm e}^{-\beta_k \, \theta^{2\, \mu_k}} \left( g_k \left( \mathbf{i} \, \theta,\dfrac{\mathbf{i} \, \theta}{n^{1/(2\, \mu_k)}} \right) -\sum_{m=0}^M \dfrac{(\mathbf{i} \, \theta)^m}{m\, ! \, n^{m/(2\, \mu_k)}} \, \dfrac{\partial^m g_k}{\partial z^m} (\mathbf{i} \, \theta,0) \right) \, {\rm d}\theta \\ =\varepsilon_{\ell,n}^{1,1}+\varepsilon_{\ell,n}^{1,2}
+\varepsilon_{\ell,n}^3 \, ,
\end{multline*}
where $\varepsilon_{\ell,n}^{1,1}$, resp. $\varepsilon_{\ell,n}^{1,2}$, corresponds to the integral on the left, resp. right, vertical segment, and $\varepsilon_{\ell,n}^3$ corresponds to the integral on the horizontal segment (see Figure \ref{fig:contour1}). We omit the dependence on $k$ of each integral for the sake of simplicity.

\begin{figure}[ht!]
\begin{center}
\begin{tikzpicture}[scale=1.25,>=latex]
\draw[black,->] (-4,0) -- (4,0);
\draw[black,->] (0,-3)--(0,0.5);
\draw[thick,black] (-3,0) -- (3,0);
\draw[thick,blue,->] (-3,0) -- (-3,-1);
\draw[thick,blue] (-3,-1) -- (-3,-2);
\draw[thick,blue,->] (-3,-2) -- (1,-2);
\draw[thick,blue] (1,-2) -- (3,-2);
\draw[thick,blue,->] (3,-2) -- (3,-0.9);
\draw[thick,blue] (3,-0.9) -- (3,0);
\draw (-3,0) node[above]{$-\delta \, n^{1/(2\, \mu_k)}$};
\draw (3,0) node[above]{$\delta \, n^{1/(2\, \mu_k)}$};
\draw (0.15,0) node[above]{$0$};
\draw (-3.5,-1.05) node[above]{{\color{red}$\varepsilon_{\ell,n}^{1,1}$}};
\draw (3.5,-1.05) node[above]{{\color{red}$\varepsilon_{\ell,n}^{1,2}$}};
\draw (-1.45,-2.8) node[above]{{\color{red}$\varepsilon_{\ell,n}^3$}};
\draw[thick,red,->] (-3.05,-1.5) arc (270:180:0.5) ;
\draw[thick,red,->] (3.05,-1.5) arc (270:360:0.5) ;
\draw[thick,red,->] (-0.9,-2.05) -- (-1.3,-2.45);
\draw (-0.35,-2) node[above]{$-\mathbf{i} \, \Xi$};
\draw (1.5,1) node {$\C$};
\node (centre) at (-3,0){$\bullet$};
\node (centre) at (3,0){$\bullet$};
\node (centre) at (-3,-2){$\bullet$};
\node (centre) at (3,-2){$\bullet$};
\end{tikzpicture}
\caption{The integration contour in the case $x_k \ge 0$ (in blue). The bullets correspond to the endpoints of the three segments that define the new contour. The initial contour is depicted in black. Each new integral appears in red.}
\label{fig:contour1}
\end{center}
\end{figure}
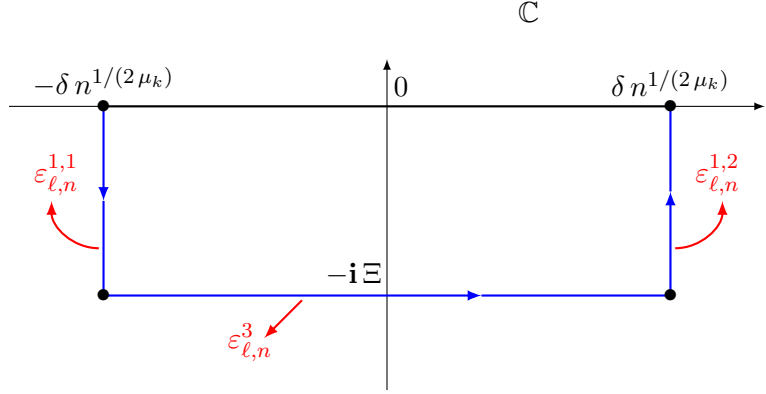

We start with the integrals on the vertical segments and compute:
\begin{multline*}
\varepsilon_{\ell,n}^{1,1} =-\mathbf{i} \, \int_0^\Xi {\rm e}^{-\mathbf{i} \, x_k \, (-\delta \, n^{1/(2\, \mu_k)}-\mathbf{i} \, u)} \, {\rm e}^{-\beta_k \, (-\delta \, n^{1/(2\, \mu_k)}-\mathbf{i} \, u)^{2\, \mu_k}} \\
\left( g_k \left( -\mathbf{i} \, \delta \, n^{1/(2\, \mu_k)}+u,-\mathbf{i} \, \delta +\dfrac{u}{n^{1/(2\, \mu_k)}} \right) -\sum_{m=0}^M \dfrac{(-\mathbf{i} \, \delta \, n^{1/(2\, \mu_k)}+u)^m}{m\, ! \, n^{m/(2\, \mu_k)}} \, \dfrac{\partial^m g_k}{\partial z^m} (-\mathbf{i} \, \delta \, n^{1/(2\, \mu_k)}+u),0) \right) \, {\rm d}u \, .
\end{multline*}
We apply the triangle inequality, use the inequalities \eqref{defbetaetoiles} and \eqref{estimgk} to get:
$$
|\varepsilon_{\ell,n}^{1,1}| \le C \, {\rm e}^{-\frac{\beta_* \, \delta^{2\, \mu_k}}{2} \, n} \, \int_0^\Xi {\rm e}^{-x_k \, u} \, {\rm e}^{2\, \beta^* \, u^{2\, \mu_k}} {\underbrace{\left| -\mathbf{i} \, \delta +\dfrac{u}{n^{1/(2\, \mu_k)}}  \right|}_{\le \sqrt{2} \, \delta}}^{M+1} \, {\rm d}u \, .
$$
We then use \eqref{estimsegment1} to get, for suitable positive constants $C$ and $c$ that do not depend on $x_k$ nor on $n$:
$$
|\varepsilon_{\ell,n}^{1,1}| \le C \, {\rm e}^{-c \, n} \, \int_0^\Xi \exp \left( -\dfrac{(2\, \mu_k-1)}{2\, \mu_k} \, x_k \, u \right) \, {\rm d}u \, .
$$
Since $x_k$ is nonnegative and $\mu_k \ge 1$, we have:
$$
|\varepsilon_{\ell,n}^{1,1}| \le C \, \Xi \, {\rm e}^{-c \, n} \, ,
$$
but since $\Xi$ is not larger than $\delta \, n^{1/(2\, \mu_k)}$, we end up with the exponential bound:
\begin{equation}
\label{estimerror1-1}
|\varepsilon_{\ell,n}^{1,1}| \le C \, {\rm e}^{-c \, n} \, ,
\end{equation}
for suitable constants $C$ and $c$ that do not depend on $\ell$ and $n$. The estimate of the integral $\varepsilon_{\ell,n}^{1,2}$ on the right vertical segment is entirely similar. Going back to \eqref{estimation2}, we use the estimate \eqref{estimerror2} as well as \eqref{estimerror1-1} in \eqref{estimerror1-init}. We have thus obtained so far the estimate:
\begin{equation}
\label{estimation3}
\left| {\boldsymbol{a}}^{\star n}_\ell -\sum_{k=1}^K \dfrac{\underline{\kappa}_k^{-\ell} F_{\boldsymbol{a}}(\underline{\kappa}_k)^n}{2 \, \pi} \int_\R {\rm e}^{-\mathbf{i} \, x_k \, \theta} \, {\rm e}^{-\beta_k \, \theta^{2\, \mu_k}} \sum_{m=0}^M \dfrac{P_{k,m}(\mathbf{i} \, \theta)}{n^{m/(2\, \mu_k)}} \, {\rm d}\theta \right| \le C \, {\rm e}^{-c \, n} +\sum_{k=1}^K n^{-1/(2\, \mu_k)}\, |\varepsilon_{\ell,n}^3| \, ,
\end{equation}
where we have used Lemma \ref{lem:expressionPkm} to simplify the integral on the left hand side of \eqref{estimation3} and we recall that the notation $\varepsilon_{\ell,n}^3$ omits the index $k$ for simplicity.

It remains to compute and estimate the integral $\varepsilon_{\ell,n}^3$ on the horizontal segment depicted on Figure \ref{fig:contour1} (for each index $k$). We have:
\begin{multline*}
\varepsilon_{\ell,n}^3 =\int_{-\delta \, n^{1/(2\, \mu_k)}}^{+\delta \, n^{1/(2\, \mu_k)}} {\rm e}^{-\mathbf{i} \, x_k \, (\theta-\mathbf{i} \, \Xi)} \, {\rm e}^{-\beta_k \, (\theta-\mathbf{i} \, \Xi)^{2\, \mu_k}} \\
\left( g_k \left( \theta-\mathbf{i} \, \Xi,\dfrac{\theta-\mathbf{i} \, \Xi}{n^{1/(2\, \mu_k)}} \right) -\sum_{m=0}^M \dfrac{(\theta-\mathbf{i} \, \Xi)^m}{m\, ! \, n^{m/(2\, \mu_k)}} \, \dfrac{\partial^m g_k}{\partial z^m} (\theta-\mathbf{i} \, \Xi,0) \right) \, {\rm d}\theta \, .
\end{multline*}
We use again the inequalities \eqref{defbetaetoiles} and \eqref{estimgk} to get:
$$
|\varepsilon_{\ell,n}^3| \le C \, \exp \left( -x_k \, \Xi +2\, \beta^* \, \Xi^{2\, \mu_k} \right) \, \int_{-\delta \, n^{1/(2\, \mu_k)}}^{+\delta \, n^{1/(2\, \mu_k)}} {\rm e}^{-\frac{\beta_*}{2} \, \theta^{2 \, \mu_k}} \,  \dfrac{|\theta -\mathbf{i} \, \Xi|^{M+1}}{n^{(M+1)/(2\, \mu_k)}} \, {\rm d}\theta \, .
$$
With our choice for $\Xi$, we obtain (see \eqref{estimsegment1}):
$$
n^{(M+1)/(2\, \mu_k)} \, |\varepsilon_{\ell,n}^3| \le C \, \exp \left( -\dfrac{(2\, \mu_k-1)}{2\, \mu_k)} \, x_k \, \Xi \right) \, \int_\R \Big( |\theta|^{M+1}+\Xi^{M+1} \Big) \, {\rm e}^{-\frac{\beta_*}{2} \, \theta^{2 \, \mu_k}} \, {\rm d}\theta \, ,
$$
and this gives:
\begin{equation}
\label{estimerror1-2}
n^{(M+1)/(2\, \mu_k)} \, |\varepsilon_{\ell,n}^3| \le C \, \Big( 1+\Xi^{M+1} \Big) \, \exp \left( -\dfrac{(2\, \mu_k-1)}{2\, \mu_k)} \, x_k \, \Xi \right) \, .
\end{equation}

Let us go back to the definition of the parameter $\Xi$ and split the final argument between the two possible regimes for $x_k$. In the first case of \eqref{defXi}, we have
\begin{multline*}
\Big( 1+\Xi^{M+1} \Big) \, \exp \left( -\dfrac{(2\, \mu_k-1)}{2\, \mu_k)} \, x_k \, \Xi \right) =
\left( 1+\dfrac{x_k^{(M+1)/(2\, \mu_k-1)}}{C} \right) \exp \left( -c \, x_k^{(2\, \mu_k)/(2\, \mu_k-1)} \right) \\
\le C \, \exp \left( -c \, x_k^{(2\, \mu_k)/(2\, \mu_k-1)} \right) \, .
\end{multline*}
In the second case of \eqref{defXi}, we have $x_k \, \Xi \ge c \, n$, and therefore:
$$
\Big( 1+\Xi^{M+1} \Big) \, \exp \left( -\dfrac{(2\, \mu_k-1)}{2\, \mu_k)} \, x_k \, \Xi \right) 
\le C \, n^{(M+1)/(2\, \mu_k)} \, {\rm e}^{-c \, n} \le C \, {\rm e}^{-c \, n} \, .
$$
Whatever the value of $x_k$, \eqref{estimerror1-2} thus gives:
$$
|\varepsilon_{\ell,n}^3| \le C \, {\rm e}^{-c \, n} +\dfrac{C}{n^{(M+1)/(2\, \mu_k)}} \, \exp \left( -c \, \left( \dfrac{|\ell-\alpha_k \, n|}{n^{1/(2\, \mu_k)}} \right)^{\frac{2\, \mu_k}{2\, \mu_k-1}} \right) \, ,
$$
and \eqref{estimation3} thus implies the estimate:
\begin{multline}
\label{estimation4}
\left| {\boldsymbol{a}}^{\star n}_\ell -\sum_{k=1}^K \dfrac{\underline{\kappa}_k^{-\ell} F_{\boldsymbol{a}}(\underline{\kappa}_k)^n}{2 \, \pi \, n^{1/(2\, \mu_k)}} \int_\R {\rm e}^{-\mathbf{i} \, x_k \, \theta} \, {\rm e}^{-\beta_k \, \theta^{2\, \mu_k}} \sum_{m=0}^M \dfrac{P_{k,m}(\mathbf{i} \, \theta)}{n^{m/(2\, \mu_k)}} \, {\rm d}\theta \right| \\
\le C \, {\rm e}^{-c \, n} +C \, \sum_{k=1}^K \dfrac{1}{n^{(M+2)/(2\, \mu_k)}} \, \exp \left( -c \, \left( \dfrac{|\ell-\alpha_k \, n|}{n^{1/(2\, \mu_k)}} \right)^{\frac{2\, \mu_k}{2\, \mu_k-1}} \right) \, ,
\end{multline}
that holds for any $\ell \in \Z$ and $n \in \N^*$. Note that we have not used the inequality $|\ell| \le L \, n$ to derive \eqref{estimation4}.

\paragraph{End of the proof.} We start from \eqref{estimation4} and use the properties of the Fourier transform to simplify the left hand side into (recall the definition \eqref{defHmubeta}):
\begin{multline}
\label{estimation5}
\left| {\boldsymbol{a}}^{\star n}_\ell -\sum_{k=1}^K \sum_{m=0}^M \dfrac{\underline{\kappa}_k^{-\ell} F_{\boldsymbol{a}}(\underline{\kappa}_k)^n}{n^{(m+1)/(2\, \mu_k)}} P_{k,m}(-{\rm d}/{\rm d}x) H_{2\, \mu_k}^{\beta_k} (x_k)  \right| \\
\le C \, {\rm e}^{-c \, n} +C \, \sum_{k=1}^K \dfrac{1}{n^{(M+2)/(2\, \mu_k)}} \, \exp \left( -c \, \left( \dfrac{|\ell-\alpha_k \, n|}{n^{1/(2\, \mu_k)}} \right)^{\frac{2\, \mu_k}{2\, \mu_k-1}} \right) \, ,
\end{multline}
where $x_k$ is defined in \eqref{defxk}. The only (minor) task is to show that the exponentially small term in $n$ can be absorbed into the generalized Gaussian functions. Actually, we shall show that this exponentially small term in $n$ is lower than any of the terms in the sum on the right hand side of \eqref{estimation5} (for instance, the first term in the sum, which corresponds to $k=1$). This is where the assumption $|\ell| \le L \, n$ is crucial. Namely, we aim at showing that for a given constant $c_0>0$, there exist positive constants $C_1$ and $c_1$ such that for $|\ell| \le L \, n$, there holds:
$$
{\rm e}^{-c_0 \, n} \le \dfrac{C_1}{n^{(M+2)/(2\, \mu_1)}} \exp \left( -c_1 \, \left( \dfrac{|\ell-\alpha_1 \, n|}{n^{1/(2\, \mu_1)}} \right)^{\frac{2\, \mu_1}{2\, \mu_1-1}} \right) \, ,
$$
which follows from choosing $C_1$ such that:
$$
\forall \, n \ge 1 \, ,\quad n^{(M+2)/(2\, \mu_1)} \, {\rm e}^{-c_0 \, n} \le C_1 \, {\rm e}^{-\frac{c_0}{2} \, n} \, ,
$$
and then by choosing $c_1$ small enough such that, for $|\ell| \le L \, n$, there holds:
$$
c_1 \, |\ell-\alpha_1 \, n|^{\frac{2\, \mu_1}{2\, \mu_1-1}} \le \dfrac{c_0}{2} \, n^{\frac{2\, \mu_1}{2\, \mu_1-1}} \, .
$$
The proof of Theorem \ref{thm1} is now complete.

\section{Consequences and examples}
\label{section4}

\subsection{Large time asymptotics for finite difference schemes}

In this Paragraph, we prove Corollary \ref{coro1}. We thus consider a sequence $\boldsymbol{a}$ that satisfies Assumptions \ref{ass:1} and \ref{ass:2}, and consider an integer $M \in \N$. We recall, see \cite{Coeuret}, that for any $\beta$ with positive real part, the function $H_{2\, \mu}^\beta$ has super-exponential decay at infinity as well as its derivatives:
$$
\forall \, N \in \N \, ,\quad \exists \, C >0 \, ,\quad \forall \, x \in \R \, ,\quad 
|H_{2\, \mu}^\beta(x)| +\cdots+|(H_{2\, \mu}^\beta)^{(N)}(x)| \le C \, \exp \left( -\dfrac{1}{C} \, |x|^{\frac{2\, \mu}{2\, \mu-1}} \right) \, .
$$
In Theorem \ref{thm1}, we can always choose the integer $L$ such that $L \ge 2\, (1+\max_k |\alpha_k|)$, so that for $|\ell| \ge L \, n$, there holds:
$$
|\ell-\alpha_k \, n| \ge \dfrac{|\ell|}{2}+n \, .
$$
For $|\ell| \ge L \, n$, the bound of Theorem \ref{thm1} on ${\boldsymbol{a}}_\ell^{\star n}$ and the above bound for $H_{2\, \mu}^\beta$ and its derivatives imply the error bound:
\begin{multline}
\label{estim-aux}
\left| {\boldsymbol{a}}^{\star n}_\ell -\sum_{k=1}^K \dfrac{\underline{\kappa}_k^{-\ell} F_{\boldsymbol{a}}(\underline{\kappa}_k)^n}{n^{1/(2\, \mu_k)}} \, H^{\beta_k}_{2 \, \mu_k} \left( \dfrac{\ell-\alpha_k n}{n^{1/(2\, \mu_k)}} \right)
-\sum_{k=1}^K \sum_{m=1}^M \dfrac{\underline{\kappa}_k^{-\ell} F_{\boldsymbol{a}}(\underline{\kappa}_k)^n}{n^{(m+1)/(2\, \mu_k)}} \, \Big( P_{k,m} (-{\rm d}/{\rm d}x) H^{\beta_k}_{2 \, \mu_k} \Big) \left( \dfrac{\ell-\alpha_k n}{n^{1/(2\, \mu_k)}} \right) \right| \\
\le C \, \exp (-c \, n -c \, |\ell|) \, ,
\end{multline}
for $|\ell| \ge L \, n$.
\bigskip

%In particular, for any $m=0,\dots,M$, the various sequences arising in the left hand side of \eqref{estim} scale as follows in $\ell^1$:
%$$
%\dfrac{1}{n^{1/(2\, \mu_k)}} \, \sum_{\ell \in \Z} \left| H^{\beta_k}_{2 \, \mu_k} \left( \dfrac{\ell-\alpha_k n}{n^{1/(2\, \mu_k)}} \right) \right| 
%+ \cdots +\dfrac{1}{n^{1/(2\, \mu_k)}} \, \sum_{\ell \in \Z} \left| \Big( P_{k,M} (-{\rm d}/{\rm d}x) H^{\beta_k}_{2 \, \mu_k} \Big) \left( \dfrac{\ell-\alpha_k n}{n^{1/(2\, \mu_k)}} \right) \right| \le C \, ,
%$$
%where the constant $C$ does not depend on $n$.

Let now ${\boldsymbol{u}}^0 \in \ell^p(\Z;\C)$ with $1 \le p \le +\infty$. Each convolution on the left hand side of \eqref{estimcoro1} is well defined and belongs to $\ell^p$ since it corresponds to the convolution on $\Z$ of an $\ell^1$ sequence with ${\boldsymbol{u}}^0 \in \ell^p$. The conclusion of Corollary \ref{coro1} follows from two observations. First of all, given any constant $c_0>0$ and and integer $\mu_k \in \N^*$, we have:
$$
\forall \, n \in \N^* \, ,\quad \dfrac{1}{n^{1/(2\, \mu_k)}} \, \sum_{\ell \in \Z} \exp \left( -c_0 \, \left( \dfrac{|\ell-\alpha_k \, n|}{n^{1/(2\, \mu_k)}} \right)^{\frac{2\, \mu_k}{2\, \mu_k-1}} \right) \le C \, ,
$$
for some appropriate constant $C$. Second, for any constant $c_1>0$, we also have:
$$
\sum_{\ell \in \Z} \exp (-c_1 \, (n+|\ell|)) \le C \, \exp (-c_1 \, n) \, .
$$
We then combine the bounds \eqref{estim} and \eqref{estim-aux} with Young's inequality to obtain the conclusion of Corollary \ref{coro1}. We have thus obtained an accurate description of the large time asymptotics of the iterated convolution ${\boldsymbol{a}}^{\star n} \star {\boldsymbol{u}}^0$. An example is detailed in the following paragraph.

\subsection{A third order scheme for the transport equation}

We report now on several calculations that can be made for the so-called $O3$ scheme, which is a finite difference approximation of the transport equation. We refer to \cite{strang2,Despres1,Despres2} for more information about such high order compact approximations.

The $O3$ scheme corresponds to the finitely supported real valued sequence ${\boldsymbol{a}}=(a_\ell)_{\ell \in \Z}$ that is defined by:
$$
a_{-1}:=\dfrac{\lambda \, (2-\lambda) \, (\lambda-1)}{6} \, ,\quad 
a_0:=\dfrac{(2-\lambda) \, (1-\lambda^2)}{2} \, ,\quad 
a_1:=\dfrac{\lambda \, (2-\lambda) \, (1+\lambda)}{2} \, ,\quad 
a_2:=-\dfrac{\lambda \, (1-\lambda^2)}{6} \, ,
$$
where $\lambda$ is a real parameter. All other values of $a_\ell$ are zero. As reported in \cite{Despres2,CF1}, the corresponding Fourier transform $F_{\boldsymbol{a}}$ can be explicitly computed, and satisfies:
$$
\forall \, \xi \in \R \, ,\quad \left| F_{\boldsymbol{a}} \big( \, {\rm e}^{\, \mathbf{i} \, \xi} \, \big) \right|^2 \, = \,  1- \frac{4}{9} \, \lambda \, (2-\lambda) \, (1-\lambda^2) \, \sin^4\left(\frac{\xi}{2}\right) \, \left( 3+4 \, \lambda \, (1-\lambda) \, \sin^2\left(\frac{\xi}{2}\right) \right) \, .
$$
In particular, for $\lambda \in (0,1)$, Assumption \ref{ass:1} is satisfied and the second possibility in Lemma \ref{lem:comportementF} occurs. Furthermore, with the notation of Lemma \ref{lem:comportementF}, there holds $K=1$ and $\underline{\kappa}_1=F_{\boldsymbol{a}}(\underline{\kappa}_1)=1$. Assuming from now on that the parameter $\lambda$ lies in the open interval $(0,1)$,  Assumption \ref{ass:2} is also satisfied with:
$$
\alpha_1=\lambda \, ,\quad \beta_1=\dfrac{\lambda \, (2-\lambda) \, (1-\lambda^2)}{24} \, \quad \mu_1=2 \, .
$$

We now explain the calculation of the cumulants at  $\underline{\kappa}_1=1$ and apply Theorem \ref{thm1} in that case. We compute, as $\xi$ tends to zero:
\begin{align*}
{\rm e}^{-\, \mathbf{i} \, \lambda \, \xi} \, F_{\boldsymbol{a}} \big( \, {\rm e}^{\, \mathbf{i} \, \xi} \, \big) =1 & -\dfrac{\lambda \, (2-\lambda) \, (1-\lambda^2)}{24} \, \xi^4 -\dfrac{\mathbf{i} \, \lambda \, (2-\lambda) \, (1-\lambda^2) \, (1-2\, \lambda)}{60} \, \xi^5 \\
&+\dfrac{\lambda \, (2-\lambda) \, (1-\lambda^2)(1-2\, \lambda+2\, \lambda^2)}{144} \, \xi^6+\dfrac{\mathbf{i} \lambda \, (2-\lambda) \, (1-\lambda^2)\, (1-2\lambda) \,(1-\, \lambda+\, \lambda^2)}{504} \, \xi^7 \\
& +O(\xi^8) \, ,
\end{align*}
which means that the power series expansion \eqref{DLk} holds at the point $\underline{\kappa}_1=1$ with:
$$
\gamma_{1,5}=-2\, \lambda \, (2-\lambda) \, (1-\lambda^2) \, (1-2\, \lambda) \, ,\quad \gamma_{1,6}=-5 \, \lambda \, (2-\lambda) \, (1-\lambda^2) \, (1-2\, \lambda+2\, \lambda^2) \, ,
$$
and
$$
\gamma_{1,7}=-10 \, \lambda \, (2-\lambda) \, (1-\lambda^2)\, (1-2\, \lambda) \,(1-\, \lambda+\, \lambda^2).
$$
Using the general formula \eqref{valeurPkm} with $m=1,2,3$, we also compute:
$$
P_{1,1}(Y)\, =\, \frac{\gamma_{1,5}}{5 \, !}\, Y^5, \quad P_{1,2}(Y) =\, \frac{\gamma_{1,6}}{6 \, !}\, Y^6+ \frac{1}{2}\left(\frac{\gamma_{1,5}}{5 \, !}\right)^2 \, Y^{10},$$
and 
$$ P_{1,3}(Y) = \, \frac{\gamma_{1,7}}{7 \, !}\, Y^7+\left(\frac{\gamma_{1,5}}{5 \, !}\right)\left(\frac{\gamma_{1,6}}{6 \, !}\right)\, Y^{11}
+\frac{1}{6}\left(\frac{\gamma_{1,5}}{5 \, !}\right)^3\, Y^{15}.
$$

\begin{figure}[t!]
\centering
\includegraphics[width=.5\textwidth]{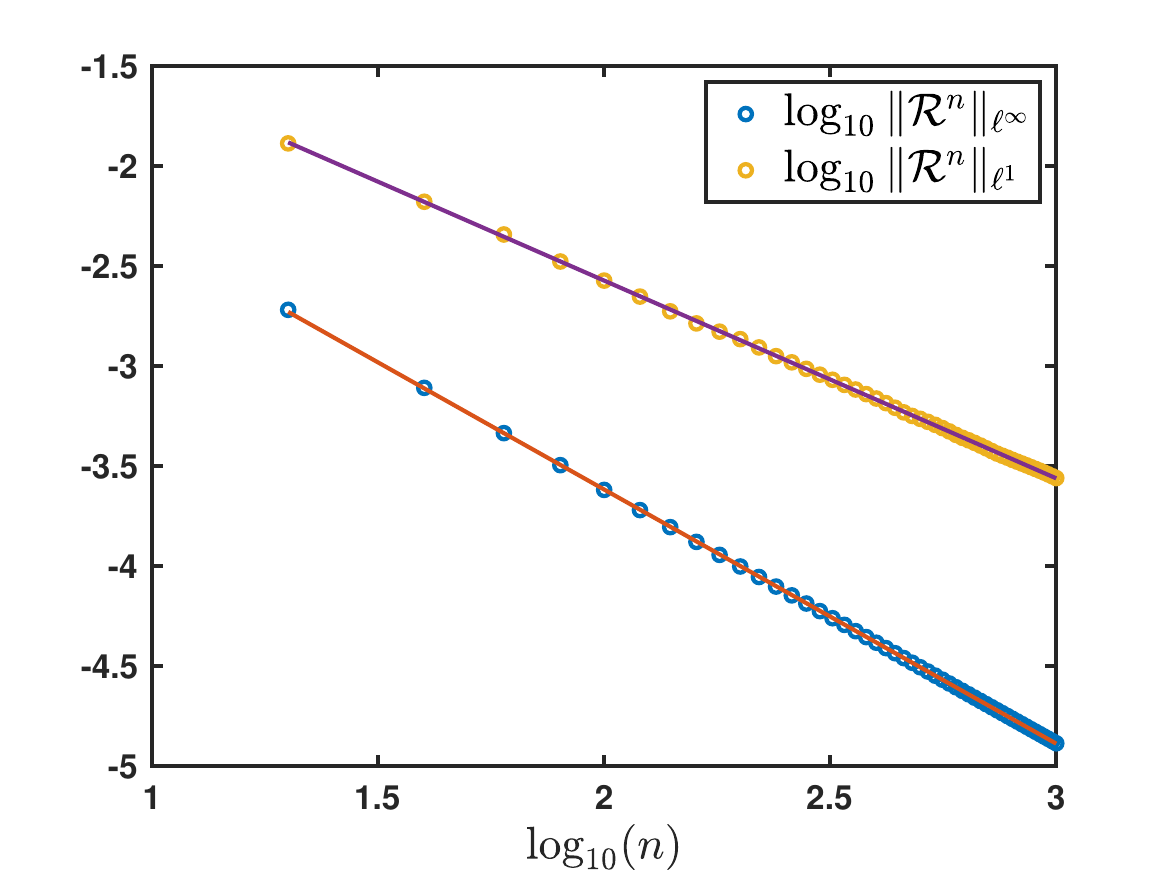}
  \caption{Illustration of the scaling factor in the generalized asymptotic expansion provided by Theorem~\ref{thm1} in the case of the $O3$ scheme. We plot $\log_{10}\|\mathcal{R}^n\|_{\ell^\infty}$ (blue circles) and $\log_{10}\|\mathcal{R}^n\|_{\ell^1}$ (orange circles) as a function of $\log_{10}(n)$ 
  together with a best linear fit for each norm for $n$ ranging from $1$ to $10^3$. For the $\ell^\infty$ norm we find a slope of $-1.2707$ while for the $\ell^1$ norm we find a slope of $-0.9887$ which compare both well with the predicted $-5/4$ and $-1$ scaling factors of Theorem~\ref{thm1}.}
  \label{fig:logplot}
\end{figure}

Next, we recall that ${\boldsymbol{a}}$ is finitely supported so the estimate \eqref{estim} holds not only for $\ell$ and $n$ in a large sector $\{ |\ell| \le L \, n \}$ but for any $(\ell,n) \in \Z \times \N^*$. Specifying from now on to $\lambda=1/2$, Theorem \ref{thm1} gives (with $M=3$ in this case):
$$
\left| {\boldsymbol{a}}^{\star n}_\ell -\dfrac{1}{n^{1/4}} \, H^{3/128}_4 \left( \dfrac{\ell-n/2}{n^{1/4}} \right) +\dfrac{1}{512 \, n^{3/4}} \, \Big( H^{3/128}_4 \Big)^{(6)} \left( \dfrac{\ell-n/2}{n^{1/4}} \right) \right| 
\le \dfrac{C}{n^{5/4}} \, \exp \left( -c \, \left( \dfrac{|\ell-n/2|}{n^{1/4}} \right)^{4/3} \right) \, .
$$
Here, we have used $\gamma_{1,5}=\gamma_{1,7}=0$ with $\gamma_{1,6}=-45/32$ so that $P_{1,1}(Y)=P_{1,3}(Y)=0$ and $P_{1,2}(Y)=-\dfrac{1}{512}Y^6$. Upon defining the sequence $\mathcal{R}^n=(\mathcal{R}^n_\ell)_{\ell\in\Z}$ from the above remainder term as
$$
\mathcal{R}^n_\ell:={\boldsymbol{a}}^{\star n}_\ell -\dfrac{1}{n^{1/4}} \, H^{3/128}_4 \left( \dfrac{\ell-n/2}{n^{1/4}} \right) +\dfrac{1}{512 \, n^{3/4}} \, \Big( H^{3/128}_4 \Big)^{(6)} \left( \dfrac{\ell-n/2}{n^{1/4}} \right)
$$ 
for any $(\ell,n) \in \Z \times \N^*$, we show in Figure~\ref{fig:logplot}  the log plot of $\|\mathcal{R}^n\|_{\ell^\infty}$ and $\|\mathcal{R}^n\|_{\ell^1}$ and recover the respective scaling $n^{-5/4}$ and $n^{-1}$. Furthermore, in Figure~\ref{fig:profiles}, we illustrate the generalized Gaussian estimate of the remainder $\mathcal{R}^n_\ell$ by showing for different time iterations that
\bqs
n^{5/4}\left|\mathcal{R}^n_\ell\right| \leq C \exp\left(- c  \, \left( \dfrac{|\ell-n/2|}{n^{1/4}} \right)^{4/3} \right),
\eqs
with constants $C=0.09$ and $c=0.225$.

Let us finally note that other examples originating from finite difference approximations of the transport equation (Lax-Friedrichs scheme and so-called $\infty$ scheme), with several tangency points ($K \ge 2$), can be found in \cite[Section 4]{CF1}, and for which the framework of Theorem~\ref{thm1} would straightforwardly apply.  

\begin{figure}[t!]
\centering
\includegraphics[width=.5\textwidth]{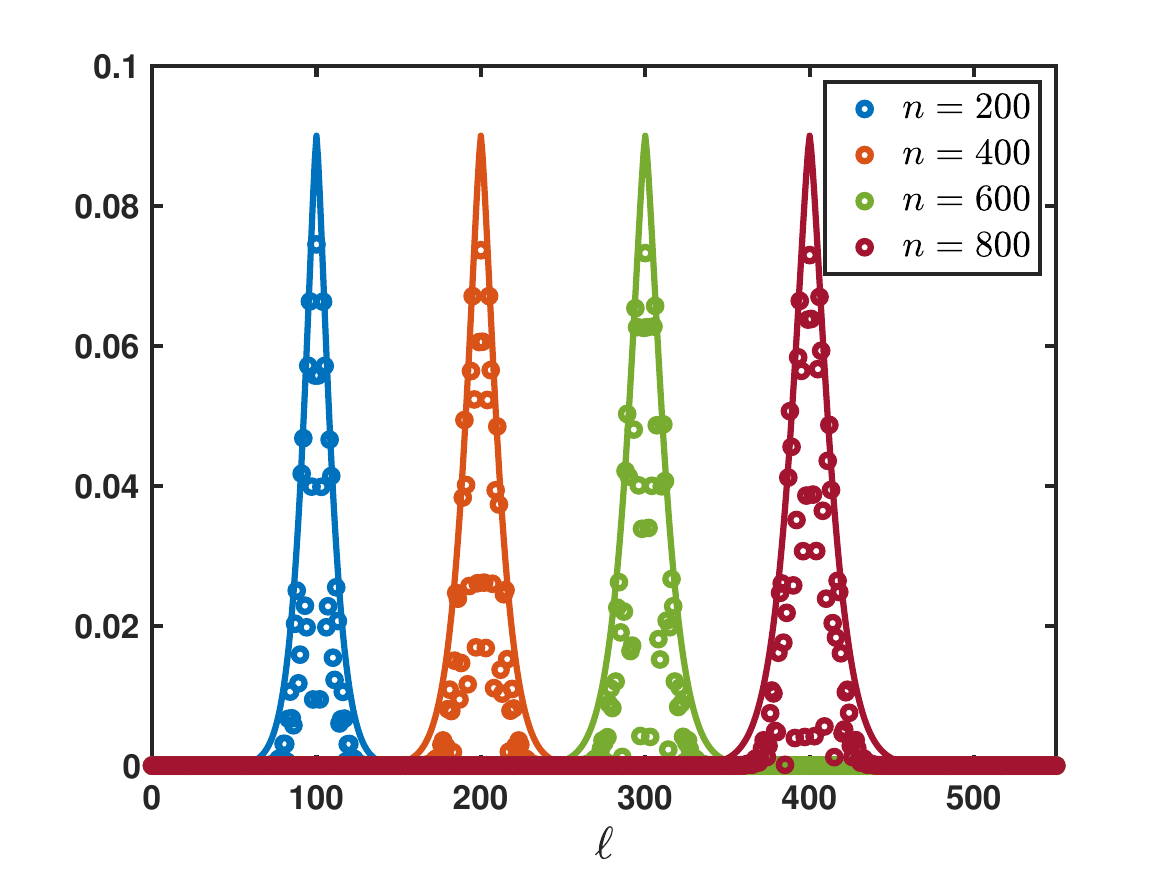}
  \caption{Illustration of the rescaled remainder term $n^{5/4}\left|\mathcal{R}^n_\ell\right|$ (colored circles) at different time iterations of the $O3$ scheme compared with a fixed generalized Gaussian profile centered at $\ell=\lambda n$ (solid lines) with $\lambda=1/2$. The fixed generalized Gaussian profile is given by the sequence $\ell \mapsto C \exp\left(- c  \, \left( \dfrac{|\ell-n/2|}{n^{1/4}} \right)^{4/3} \right)$ with constants $C=0.09$ and $c=0.225$.}
  \label{fig:profiles}
\end{figure}

\bibliographystyle{alpha}
\bibliography{CF}
\end{document}